\documentclass[a4paper,leqno]{amsart}

\usepackage{amsmath,amssymb,amsfonts,amsthm,mathrsfs}

\usepackage{hyperref}
\usepackage{pdfsync}
\usepackage{color}

\newcommand{\R}{\mathbb R}
\newcommand{\N}{\mathbb N}

\newcommand{\dis}{\displaystyle}

\def\({\left(}
\def\){\right)}
\def\<{\left\langle}
\def\>{\right\rangle}

\def\O{\mathcal O}

\def\d{{\partial}}
\def\eps{\varepsilon}

\DeclareMathOperator{\IM}{Im}

\def\Eq#1#2{\mathop{\sim}\limits_{#1\rightarrow#2}}
\def\Tend#1#2{\mathop{\longrightarrow}\limits_{#1\rightarrow#2}}

\theoremstyle{plain}
\newtheorem{theorem}{Theorem} [section]
\newtheorem{lemma}[theorem]{Lemma}

\newtheorem{proposition}[theorem]{Proposition}

\theoremstyle{remark}
\newtheorem{remark}[theorem]{Remark}

\theoremstyle{definition}
\newtheorem{definition}[theorem]{Definition}

\newtheorem{assumption}[theorem]{Assumption}
\newtheorem*{conjecture}{Conjecture}

\def\Eq#1#2{\mathop{\sim}\limits_{#1\rightarrow#2}}
\def\Tend#1#2{\mathop{\longrightarrow}\limits_{#1\rightarrow#2}}

\numberwithin{equation}{section}

\begin{document}

\title[Cubic-quintic Schr\"odinger equation]{Orbital stability vs. scattering in 
  the cubic-quintic Schr\"odinger equation}  

\author[R. Carles]{R\'emi Carles}
\address{Univ Rennes, CNRS\\ IRMAR - UMR 6625\\ F-35000
  Rennes, France}
\email{Remi.Carles@math.cnrs.fr}

\author[C. Sparber]{Christof Sparber}
\address{Department of Mathematics, Statistics, and
  Computer Science\\
  M/C 249\\
University of Illinois at Chicago\\
851 S. Morgan Street
Chicago\\ IL 60607, USA} 
\email{sparber@uic.edu}

\begin{abstract}
 We consider the cubic-quintic nonlinear Schr\"odinger equation in space
 dimension up to three. The cubic nonlinearity is thereby focusing while the
 quintic one is defocusing, ensuring global
 well-posedness of the Cauchy problem in the energy space. The main
 goal of this paper is to investigate the interplay between dispersion and orbital (in-)stability of solitary waves. 
 In space dimension one, it is already known that all solitons
 are orbitally stable. In dimension two, we show that if the 
 initial data belong to the conformal space, and have at most the mass of the ground state
 of the cubic two-dimensional Schr\"odinger equation, then the solution is 
 asymptotically linear. For larger mass, solitary wave solutions 
 exist, and we review several results on their stability. Finally, in dimension
 three, relying on previous 
 results from other authors, we show that solitons may or may not
 be orbitally stable. 
\end{abstract}
\thanks{RC is supported by Rennes M\'etropole through its AIS program. CS acknowledges support by the NSF through grant no. DMS-1348092}
\maketitle


\section{Introduction and main results}
\label{sec:intro}


\subsection{Basic setting} We consider the nonlinear Schr\"odinger equation (NLS) with competing 
cubic-quintic nonlinearities,
\begin{equation}
  \label{eq:nls}
  i\d_t u +\frac{1}{2}\Delta u =-|u|^2u+|u|^4u,\quad x\in \R^d,
\end{equation}
in space dimension $d\le 3$. The quintic nonlinearity was introduced
in \cite{Pushkarov}, and adopted
in several physical situations: typically in optics (see
e.g. \cite{Kivshar03}), or in Bose-Einstein
condensation (e.g. \cite{PhysRevA63,JPhysB,PhysRevE}). We refer to the
review \cite{Malomed19} for more precise references. In particular,
the incorporation of the defocusing quintic term is motivated by the
stabilization of two- and three-dimensional vortex solitons. 

Recall some of basic features of this nonlinearity in terms of
criticality for the Cauchy problem. Depending on the space dimension, the NLS is seen to be:
\smallbreak

$\bullet$ $d=1$: focusing $L^2$-subcritical plus defocusing $L^2$-critical (and $H^1$-subcritical).

$\bullet$ $d=2$:  focusing $L^2$-critical plus defocusing
  $L^2$-supercritical  (and $H^1$-subcritical).
  
$\bullet$ $d=3$: focusing $L^2$-supercritical plus defocusing
  $H^1$-critical.
\smallbreak

It is already known from the case of more general, gauge-invariant nonlinearities (see
e.g. \cite{CazCourant}), that equation \eqref{eq:nls} formally enjoys three basic conservation laws, namely:
\smallbreak

$\bullet$ Mass: $\dis
  M(u)=\| u(t, \cdot) \|_{L^2(\R^d)}^2,$

$\bullet$  Angular momentum: $\dis
 J(u)=\IM\int_{\R^d}\bar u(t,x)\nabla u(t,x)dx,$

$\bullet$ Energy:
$\dis  E(u) = \frac{1}{2}\|\nabla u(t, \cdot)\|_{L^2(\R^d)}^2 -\frac{1}{2}\|
u(t, \cdot)\|_{L^4(\R^d)}^4+ \frac{1}{3}\|  u(t, \cdot)\|_{L^6(\R^d)}^6.$
\smallbreak

In dimensions $2$ and $3$, an effect of the quintic term
is to {\it prevent finite time blow-up} which may occur in the purely cubic
case (cf. \cite{CazCourant}). Indeed, the conservation of the energy, combined with H\"older's inequality,
\begin{equation}\label{eq:holder}
  \|u \|_{L^4(\R^d)}^4 \le \|u \|_{L^2(\R^d)}\|u \|_{L^6(\R^d)}^3, 
\end{equation}
shows that the cubic focusing part cannot be an obstruction to global
well-posedness, at least in $H^1$. For $d\le 2$, global well-posedness
then follows from classical results (see e.g. \cite{CazCourant}). For $d=3$, we
refer to \cite{Zhang06}, as the quintic term is energy-critical. 

\begin{proposition}[Global well-posedness]\label{prop:GWP}
  Let $d\le 3$. For any $u_0\in H^1(\R^d)$, \eqref{eq:nls} has a
  unique global solution $u\in C(\R;H^1(\R^d))$ such that $u_{\mid
    t=0}=u_0$. The solution obeys the conservation of mass, energy, and
  momentum. If in addition
  \begin{equation*}
    u_0\in \Sigma:=\left\{f\in H^1(\R^d),\ x\mapsto |x |f \in L^2(\R^d)\right\},
  \end{equation*}
  then $u \in C(\R;\Sigma)$. 
\end{proposition}
Numerically, one observes a kind of oscillatory behavior within the solution $u$, which is due to the 
competition of focusing and defocusing effects within \eqref{eq:nls}, cf. \cite{SuSu} for more details. 

\begin{remark} Recall that in dimension $d=2$ or $3$, the quintic term is
$L^2$-supercritical, so we cannot hope to solve the Cauchy problem at
this regularity level. Moreover, since the cubic-quintic nonlinearity stems
from physics, it is sensible to work in $H^1$, where the energy is
well-defined.
\end{remark}

Complementing the case of prescribed initial data, we may also want to prescribe
asymptotic states (or scattering states) and an asymptotically linear
behavior, provided that $d\ge 2$. 
We thereby recall that in the case $d=1$, the cubic nonlinearity causes long-range effects, and no non-trivial solution to
\eqref{eq:nls} can be asymptotically linear, cf. \cite{Barab}. However, in dimensions $d=2, 3$ one can rely on classical techniques (see e.g. \cite{CazCourant}) 
or the results of \cite{Zhang06}, respectively, to obtain:

\begin{proposition}[Scattering]\label{prop:scattering}
 Let $d=2$ or $3$. For any $u_-\in H^1(\R^d)$, \eqref{eq:nls} has a
  unique global solution $u\in C(\R;H^1(\R^d))$ such that
  \begin{equation*}
    \left\|u(t, \cdot)-e^{i\frac{t}{2}\Delta}u_-\right\|_{H^1(\R^d)}\Tend t {-\infty}0.
  \end{equation*}
  In particular,
  \begin{equation*}
    M(u)=\|u_-\|^2_{L^2},\quad E(u)=\frac{1}{2}\|\nabla
    u_-\|_{L^2}^2,\quad \forall t\in \R.
  \end{equation*}
  If in addition $u_-\in \Sigma$, then $u\in C(\R;\Sigma)$ and
  \begin{equation*}
       \left\|e^{-i\frac{t}{2}\Delta}u(t, \cdot)-u_-\right\|_{\Sigma}\Tend t {-\infty}0.
  \end{equation*}
\end{proposition}
We recall that $e^{i\frac{t}{2}\Delta}$ is unitary on $H^1(\R^d)$, but
not on $\Sigma$ (see e.g. \cite{CazCourant}), hence the final formula
above.
\smallbreak

As in the case with purely cubic nonlinearity, not every finite-energy solution of \eqref{eq:nls} is necessarily asymptotically
linear. Finite time blow-up is of course ruled out in our case, but time-periodic {solitary
wave solutions} also exist. 

\begin{definition}\label{def:soliton}
  A \emph{standing wave} or \emph{soliton} of \eqref{eq:nls} is a
  solution of the form $e^{i\omega t}\phi(x)$, with $\omega\in \R$ and
  $\phi$ satisfying
  \begin{equation}\label{eq:solitondef}
    -\frac{1}{2}\Delta \phi + \omega \phi
    -|\phi|^2\phi+|\phi|^4\phi=0,\quad \phi\in
    H^1(\R^d)\setminus\{0\}. 
  \end{equation}
  The associated action is given by
  \begin{equation*}
    S(\phi) = \frac{1}{2}\|\nabla \phi\|_{L^2}^2 +\omega
    \|\phi\|_{L^2}^2 -\frac{1}{2}\|\phi\|_{L^4}^4+\frac{1}{3}\|\phi\|_{L^6}^6.
  \end{equation*}
 A solution $\phi$ is a \emph{ground state} if $S(\phi)\le S(\varphi)$ for any solution
 $\varphi$ of \eqref{eq:solitondef}.
\end{definition}
As we will see in Section~\ref{sec:soliton}, if $d\le 3$, \eqref{eq:solitondef}
admits a solution if and only if
\begin{equation*}
  0<\omega<\tfrac{3}{16}. 
\end{equation*}
It turns out that for $d=1$, explicit solitary wave solutions are available for this
range of $\omega$, see below. In the present paper, we will review (and expand on) several results about the (in-)stability 
of solitary waves, a question which is closely related to dispersive effects in
\eqref{eq:nls}. Due to the invariants of the equations (in our case,
translation in space and multiplication by $e^{i\theta}$ for a
constant $\theta$), it is customary to consider {\it orbital
stability}, for which two approaches are available in the case of
nonlinear Schr\"odinger equations: the first one historically, due to
Cazenave and Lions \cite{CaLi82}, consists in showing that the set of
energy minimizers, subject to a mass constraint, is stable under the flow of
the equation. In some cases (typically, when the nonlinearity is
homogeneous, as well as for the logarithmic nonlinearity \cite{Ar16,Caz83}), one is able to describe this set in more detail. The other one,  known as
Grillakis-Shatah-Strauss theory, was introduced in \cite{GSS87} (see
also \cite{BGR15}), and generalized the ideas developed by
M.~Weinstein in \cite{Weinstein85,Weinstein86CPAM}. This approach has proven particularly
useful in the case of homogeneous nonlinearities
(or asymptotically
homogeneous ones, see \cite{BGR15} and references therein) 
and in space dimension one (e.g. \cite{IlievKirchev93}). We will collect several results on both of these approaches (depending on $d=1,2,3$), 
and accordingly introduce the following two notions of
orbital stability.

\begin{definition}\label{def:set-stability}
  For  $\rho>0$, denote
  \begin{equation*}
    \Gamma(\rho) = \left\{ u\in H^1(\R^d),\  M(u)=\rho\right\},
  \end{equation*}
and assume that the minimization problem
  \begin{equation}\label{eq:8.3.5}
u\in \Gamma(\rho),\quad E(u)=    \inf \{ E(v)\ ;\  v\in \Gamma(\rho)\} 
  \end{equation}
 has a solution. Denote by $\mathcal E(\rho)$ the set of such
 solutions. We say that solitary waves are $\mathcal E(\rho)$-{\it orbitally
 stable}, if for all $\eps>0$, there exists $\delta>0$ such that if
 $u_0\in H^1(\R^d)$ satisfies
 \[\inf_{\phi\in \mathcal E(\rho)}\|u_0-\phi\|_{H^1(\R^d)}\le \delta,\]
  then the
  solution to \eqref{eq:nls} with $u_{\mid t=0}=u_0$ satisfies
  \begin{equation*}
    \sup_{t\in \R}\inf_{\phi\in \mathcal E(\rho)}\left\|u(t,
      \cdot)-\phi\right\|_{H^1(\R^d)}\le \eps. 
  \end{equation*}
\end{definition}
We note that if $\phi\in \mathcal E(\rho)$, then 
\[
\{ e^{i\theta}
\phi(\cdot -y);\ \theta\in\R,\ y\in \R^d\}\subset \mathcal
E(\rho).
\] 
When the nonlinearity is homogeneous (and
$L^2$-subcritical), this inclusion becomes an equality, see
\cite{CaLi82,CazCourant}. In this case, the above notion meets the following
one, which is stronger, in general: 
\begin{definition}\label{def:stability}
  Let $\phi$ be a solution of \eqref{eq:solitondef}. The standing wave
  $e^{i\omega t}\phi(x)$ is  orbitally stable in $H^1(\R^d)$, if
  for all $\eps>0$, there exists $\delta>0$ such that if $u_0\in
  H^1(\R^d)$ satisfies
  \[\|u_0-\phi\|_{H^1(\R^d)}\le \delta,\]
  then the
  solution to \eqref{eq:nls} with $u_{\mid t=0}=u_0$ satisfies
  \begin{equation*}
    \sup_{t\in \R}\inf_{{\theta\in \R}\atop{y\in
      \R^d}}\left\|u(t, \cdot)-e^{i\theta}\phi(\cdot
      -y)\right\|_{H^1(\R^d)}\le \eps.
  \end{equation*}
  Otherwise, the standing wave is said to be unstable. 
\end{definition}
We emphasize that for ``truly'' non-homogeneous nonlinearities in dimensions $d\ge 2$ {\it only} $\mathcal E(\rho)$-orbitally 
stability is known; see e.g. \cite{CoJeSq10,Shibata2014}. Moreover, it is not clear in general that 
ground states are members of $\mathcal E(\rho)$.


\subsection{One-dimensional case}

In the case $d=1$, the overall picture is very neat. Firstly, for
$0<\omega<\tfrac{3}{16}$, solutions to \eqref{eq:solitondef} are
given by (\cite{Pushkarov}, see also \cite{Cowan})
\begin{equation}\label{eq:phi1D}
  \phi(x) =
  2\sqrt{\frac{\omega}{1+\sqrt{1-\tfrac{16\omega}{3}}\cosh\(2x\sqrt{2\omega}\)}}. 
\end{equation}
Note that in view of \cite{BL83a}, this real-valued solution is unique, up to
translation and change of sign.
The orbital stability of these nonlinear ground states was established in
\cite[Theorem~3, case (1)]{Ohta95}. 
\begin{proposition}[Orbital stability in 1D]
  Let $d=1$, and $0<\omega<\tfrac{3}{16}$. The solitary wave $e^{i\omega
    t}\phi(x)$, where $\phi$ is given by \eqref{eq:phi1D}, is
  orbitally stable. 
\end{proposition}
The proof of this result combines the well-known Grillakis-Shatah-Strauss criterion \cite{GSS87} with the analysis of \cite{IlievKirchev93} and an explicit
formula for second order ODEs without first order derivatives, a strategy which
seems to be restricted to the 1D case and not suited for solutions to
\eqref{eq:solitondef} in $d\ge 2$.
\smallbreak

We note that in \cite{Ohta95}, more general nonlinearities are
considered, including the following generalization of \eqref{eq:nls}:
\begin{equation*}
  i\d_t u +\frac{1}{2}\d_x^2 u = -|u|^{p-1}u+|u|^{q-1}u,\quad q>p. 
\end{equation*}
It is shown that if $p\le 5$, then all ground states are orbitally
stable. On the other hand, if $p>5$, ground states with
$\omega>0$ sufficiently small become unstable, while if $\omega$ is sufficiently large
 (but not too large, since ground states have to exist), they remain stable. We note that the value $p=5$
corresponds to an $L^2$-{\it critical}, focusing nonlinearity. It is therefore
natural to expect that in the case of the cubic-quintic nonlinearity
\eqref{eq:nls}, all ground states are orbitally stable when $d=2$,
while in $d=3$ some will be stable and others unstable. We will give
several pieces of rigorous evidence supporting this heuristics. 


\subsection{Two-dimensional case}

We now turn to the case $d=2$ and recall that the results of \cite{TaoVisanZhang07} show that for $\|u_0\|_{L^2}$ {\it sufficiently small}, the
  solution to \eqref{eq:nls} is asymptotically linear. It turns out
  that since the cubic term is $L^2$-critical in 2D, we can in fact be more precise.
  
To this end, let $Q$ be the {\it cubic nonlinear ground state}, i.e., the unique positive radial solution to 
\begin{equation}
  \label{eq:Q2D}
  -\frac{1}{2}\Delta Q+Q -Q^3=0,\quad x\in \R^2.
\end{equation}
 In view of \cite{Weinstein83}, and noting that we have an extra
  factor $\frac{1}{2}$ in front of the Laplacian in \eqref{eq:Q2D}
  compared to \cite{Weinstein83}, the sharp 
  Gagliardo-Nirenberg inequality reads
  \begin{equation}
    \label{eq:GNsharp}
    \|u\|_{L^4(\R^2)}^4\le
    \(\frac{\|u\|_{L^2(\R^2)}}{\|Q\|_{L^2(\R^2)}}\)^2\|\nabla
    u\|_{L^2(\R^2)}^2,\quad \forall u\in H^1(\R^2). 
  \end{equation}
In the focusing cubic case, i.e., without
  the quintic term, we know from 
  \cite{Dodson15} that if $u_0\in L^2$ and  $\|u_0\|_{L^2}< \|Q\|_{L^2}$,
  global existence and scattering hold (see also \cite{KTV09} for the
  case of radial data $u_0$). In the presence of \eqref{eq:nls}, it
  was proved in \cite{Cheng-p} that for $u_0\in H^1$
with $\|u_0\|_{L^2}< \|Q\|_{L^2}$, scattering holds as well, relying
on the cubic case from  \cite{Dodson15}.
In our first main result below, we shall show that 
  the effect of the additional quintic term
  is not only to guarantee global well-posedness, but also to extend
  this dispersive result to the $L^2$-sphere $\{
  \|u_0\|_{L^2}= \|Q\|_{L^2}\}$. We emphasize that we assume $u_0\in
  \Sigma$, not simply $u_0\in H^1(\R^2)$, see 
  Section~\ref{sec:dispersive} for a more 
  precise discussion of this aspect.

\begin{theorem}[Mass (sub-)critical scattering in 2D]\label{theo:2Ddispersion}
  Let $d=2$. If $u_0\in \Sigma$ with 
  \begin{equation*}
    \|u_0\|_{L^2}\le \|Q\|_{L^2},
  \end{equation*}
  then the solution $u\in C(\R;\Sigma)$ to \eqref{eq:nls} such that $u_{\mid t=0}
  =u_0$ is asymptotically linear, i.e. there exist $u_\pm\in \Sigma$
  such that
  \begin{equation*}
    \|e^{-i\frac{t}{2}\Delta}u(t, \cdot) -u_\pm\|_{\Sigma}\Tend t {\pm \infty}0.
  \end{equation*}
\end{theorem}
On a heuristic level, we may argue in the same fashion as in 
\cite{HoRo08}, and recall that the standard virial computation for \eqref{eq:nls}
yields, 
\begin{equation*}
  \frac{d^2}{dt^2}\int_{\R^2}|x|^2|u(t,x)|^2dx = 2E(u)+\frac{2}{3}\|u(t)\|_{L^6(\R^2)}^6\ge 2E(u_0),
\end{equation*}
where $E(u)=E(u_0)$ is the conserved energy. In view of the sharp
Gagliardo-Nirenberg inequality, we have, under the assumptions of
Theorem~\ref{theo:2Ddispersion},
\begin{equation*}
  \frac{d^2}{dt^2}\int_{\R^2}|x|^2|u(t,x)|^2dx \ge \frac{2}{3}\|u_0\|_{L^6(\R^2)}^6.
\end{equation*}
The time-derivative of the virial of $u$ is therefore increasing, a first hint that the solution is dispersive. In order to make this
statement rigorous, especially in the limiting case
$\|u_0\|_{L^2}= \|Q\|_{L^2}$, we rely on a conformal transform,
and rigidity results regarding the concentration phenomenon in
nonlinear Schr\"odinger equations, see Section \ref{sec:mcrit} below. 

Our second main result concerns the stability of solitary waves:
\begin{theorem}[Nonlinear ground states in 2D]\label{theo:2Dstability}
    Let $d=2$. Then, for all $\omega\in ]0,\tfrac{3}{16}[$, there exists a solitary
    wave solution $u(t,x) = e^{i\omega t}\phi_\omega(x)$ to \eqref{eq:nls}. In addition, we have:
    \begin{enumerate}
    \item For any $\rho>\|Q\|_{L^2}^2$, there exists  a ground state such
      that $\|\phi_\omega\|_{L^2}^2=\rho$. 
      \item The ground state solution is unique, up to translation and
        multiplication by $e^{i\theta}$, for constant $\theta\in \R$.
    \item There exists $0<\omega_0\le \omega_1\le \tfrac{3}{16}$ 
      such that for $\omega \in ]0,\omega_0[\cup ]
      \omega_1,\tfrac{3}{16}[$, $\phi_\omega$  is orbitally stable. 
   \item For any  $\rho>\|Q\|_{L^2}^2$, the set $\mathcal E(\rho)$
       is non-empty and solitary waves are $\mathcal
       E(\rho)$-orbitally stable. 
    \end{enumerate}
  \end{theorem}
  We emphasize the fact that for any mass {\it strictly larger} than that
  of the cubic ground state $Q$, we can find a soliton of the cubic-quintic NLS, while for a
  mass less or equal to that of $Q$, all solutions to \eqref{eq:nls} are asymptotically
  linear. This is in sharp contrast with the analogous situation in the case of a single pure power nonlinearity, where the critical sphere (in $L^2$ or other homogeneous
  Sobolev spaces) {\it always contains non-dispersive elements}, see
  e.g. \cite{DuHoRo08,KeMe06,Keraani06,TVZ08}.
  \smallbreak

  In view of numerical experiments showing that $\omega\mapsto
  M(\phi_\omega)$ is an increasing map, and of the analogy with the
  one-dimensional case, it is natural to conjecture that $ \omega_0
  =\tfrac{3}{16}$, that is, all ground states are orbitally stable, see also
  \cite{LewinRotaNodari-p}. Moreover, we expect that all ground states are energy minimizers.

  
\subsection{Three-dimensional case}

In $d=3$, equation \eqref{eq:nls} has already been studied in \cite{KOPV17}. However, no statement concerning the (in-)stability of solitary waves is given in there. 
Here, we shall state the following proposition, the proof of which relies on elements already present in \cite{KOPV17}:
\begin{proposition}[Soliton (in-)stability in 3D]\label{prop:soliton3D}
  Let $d=3$. For all $\omega\in ]0,\tfrac{3}{16}[$, there exists a ground
  state solution $\phi_\omega$ which is unique, up to translation and
  multiplication by $e^{i\theta}$, for constant $\theta\in \R$. Moreover:
\begin{enumerate}
    \item There exists  $\rho_0>0$ such that $M(\phi_\omega)\ge
      \rho_0$ for all $\omega  \in ]0,\tfrac{3}{16}[$.
    \item If $u_0\in H^1(\R^3)$ is such that $M(u_0)<\rho_0$, then
      the solution $u\in C(\R;H^1(\R^3))$ to \eqref{eq:nls} with
      $u_{\mid t=0}=u_0$ is asymptotically linear, i.e. there exists
      $u_\pm\in H^1(\R^3)$ such that
      \begin{equation*}
        \|e^{-i\frac{t}{2}\Delta} u(t,\cdot)-  u_\pm\|_{H^1}=\|u(t,\cdot)- e^{i\frac{t}{2}\Delta} u_\pm\|_{H^1}  \Tend t
        {\pm \infty} 0.
      \end{equation*}
    \item There exists $0<\omega_0<\tfrac{3}{16}$ such that for all
      $0<\omega<\omega_0$, $\phi_\omega$ is unstable.
      \item There exists $\omega_0\le \omega_1<\tfrac{3}{16} $ such that for all
      $\omega_1<\omega<\tfrac{3}{16}$, $\phi_\omega$ is orbitally
      stable.
      \item There exists $\rho_1>\rho_0$ such that for $\rho\ge \rho_1$,
solitary waves are  $\mathcal E(\rho)$-orbitally stable.  
    \end{enumerate}
  \end{proposition}
 
The minimal mass $\rho_0$ is related to a specific ground state $\phi$, which,
unlike in the 2D case, cannot be directly described as the solution to
some differential equation, but rather as the optimizer of a suitable Weinstein
functional. More precisely, $\rho_0= M(\phi)$ where
\begin{equation*}
  \phi= \inf_{u\in
    H^1(\R^3)\setminus\{0\}}\frac{\|u\|_{L^2}\|u\|_{L^6}^{3/2}\|\nabla
    u\|_{L^2}^{3/2}}{\|u\|_{L^4}^4} ,
\end{equation*}
see \cite{KOPV17} for more details (and, in particular, for the proof of items (1) and (2) in the last proposition). 
One again expects the equality $\omega_0=\omega_1$ to hold. More precisely,
Conjecture~2.3 from \cite{KOPV17} (see also \cite{LewinRotaNodari-p}),
which is supported by numerics, states: 
\begin{conjecture}
  There exists $0<\omega_*<\tfrac{3}{16} $ so that $\omega\mapsto M(\phi)$ is
  strictly decreasing for $\omega<\omega_*$, and strictly increasing
  for $\omega>\omega_*$. 
\end{conjecture}
If this indeed holds true, one can take $\omega_0=\omega_1=\omega_*$ in
Proposition~\ref{prop:soliton3D}.
\smallbreak

One may also wonder about the precise nature of instability. Recall, that in the
case of a {\it single} power nonlinearity, instability is always due to the possibility of finite-time blow-up (see e.g. \cite{CazCourant} and
references therein). Very recently, Fukuya and Hayashi \cite{FuHa-p} 
have established instability results for NLS with a double power
nonlinearity, but in their work the {\it focusing term dominates} the defocusing one (thereby extending the results of \cite{CMZ16}). 
They rely on the possibility of 
blow-up or invoke the 
Grillakis-Shatah-Strauss theory, in which case the nature of the
instability still remains unclear. For nonlinearly coupled \emph{systems} of NLS, Correia, Oliveira and Silva
\cite{CoOlSi-p}  have shown that instability may correspond to a
transfer of mass from one equation to the other. None of these former results, however, apply to our situation. 

In our case, one may expect that the
stable manifolds analyzed in \cite{KriegerSchlag06,Schlag09} become
open neighborhoods, in the sense that in a full neighborhood of the
unstable ground state (not only in a manifold with limited
co-dimension), the solution $u$ bifurcates from the solitary wave
$e^{i\omega t}\phi(x)$, yielding a behavior of the for
\begin{equation*}
  u(t,x) = W(t,x) + e^{i\frac{t}{2}\Delta}u_+(x) + o_{L^2}(1)\quad
  \text{as }t\to\infty,
\end{equation*}
for some $u_+\in L^2(\R^3)$, and where $W$ is a (possibly different)
ground state, modulated by a moving set of parameters (see
\cite{KriegerSchlag06,Schlag09} for details). Typically, if $W=0$, the
solution becomes fully dispersive, while if $u_+=0$, $u$ behaves like a
solitary wave. 
\smallbreak

The rest of this paper is now organized as follows: In
Section~\ref{sec:dispersive}, we prove
Theorem~\ref{theo:2Ddispersion}. In Section~\ref{sec:soliton}, we
analyze some general results on solitary waves for \eqref{eq:nls}. In
Section~\ref{sec:further}, we recall further properties of the ground
states related to the Grillakis-Shatah-Strauss theory. These will imply all
the points in Theorem~\ref{theo:2Dstability} except the fourth one, as well as 
items (3) and (4) of Proposition~\ref{prop:soliton3D}. The
proof of the remaining statements on $\mathcal E(\rho)$-stability within
Theorem~\ref{theo:2Dstability} and
Proposition~\ref{prop:soliton3D} is given in
Section~\ref{sec:cazenave-lions}.

\subsubsection*{Acknowledgements} We are grateful to Mathieu Lewin and
Simona Rota Nodari for stimulating discussions and for an early view
of their results from \cite{LewinRotaNodari-p}. 


\section{Dispersive behavior in 2D}
\label{sec:dispersive}

\subsection{Space-time norms} In this section, our main goal is to prove Theorem~\ref{theo:2Ddispersion}. Recall that
for two-dimensional Schr\"odinger equation, a {\it Strichartz-pair} $(q,r)$ is
\emph{admissible} if
\begin{equation*}
  \frac{2}{q}+\frac{2}{r} =1,\quad 2\le r<\infty.
\end{equation*}
We denote by
\begin{equation*}
  \|u\|_{S(I)} = \sup_{(q,r)\text{ admissible}}\|u\|_{L^q(I;L^r(\R^2))}.
\end{equation*}
In view of \cite[Theorem~1.3]{TaoVisanZhang07}, to prove
Theorem~\ref{theo:2Ddispersion} with $\Sigma$ replaced by the larger
space $H^1(\R^2)$, it suffices to prove that for any $u_0\in
H^1(\R^2)$ with $\|u_0\|_{L^2}\le \|Q\|_{L^2}$, the global solution $u$
provided by Proposition~\ref{prop:GWP} satisfies
\begin{equation}\label{eq:dispcond}
\|u\|_{S(\R)}+\|\nabla u\|_{S(\R)}<\infty.
\end{equation}
\begin{remark}Note that Theorem~\ref{theo:2Dstability} contains the particular information that one can
find $u_0\in H^1(\R^2)$ with $\|u_0\|_{L^2}- \|Q\|_{L^2}>0$ 
arbitrarily small, such that $\|u\|_{S(\R)}=\infty$.
\end{remark}

As a first, basic step, we show that \eqref{eq:dispcond} can be reduced to the following:

\begin{lemma}[Reduction step]\label{lem:reduc}
  Let $d=2$ and $u_0\in H^1(\R^2)$. If the global solution provided by
  Proposition~\ref{prop:GWP} satisfies 
  $\|u\|_{S(\R)}<\infty$, then we also have
  \[ \|\nabla u\|_{S(\R)}<\infty,\]
  and so $u$ is asymptotically linear,
  \[ \exists u_\pm \in H^1(\R^2),\quad
    \left\|u(t)-e^{i\frac{t}{2}\Delta}u_\pm\right\|_{H^1(\R^2)} \Tend t {\pm
      \infty} 0.\]
  If in addition $u_0\in \Sigma$, then $u_\pm \in \Sigma$ and
  \[ 
    \left\|e^{-i\frac{t}{2}\Delta}u(t)-u_\pm\right\|_{H^1(\R^2)} \Tend t {\pm
      \infty} 0.\]
\end{lemma}
\begin{proof}
  From \cite{TaoVisanZhang07}, we only have to check that
  $\|u\|_{S(\R)}<\infty$ implies $\|\nabla u\|_{S(\R)}<\infty$.
  Let $I=[t_0,t_1]$ be
some time interval, with $t_1\ge t_0\ge 0$ to simplify
notations. 
Considering the Duhamel's formula associated to \eqref{eq:nls}, taking
the gradient and applying Strichartz estimates, we find 
\begin{equation*}
 \|\nabla u\|_{S(I)}\lesssim \|\nabla u(t_0)\|_{L^2} + \left\|
   u^2\nabla u\right\|_{L^{4/3}(I\times \R^2)} + \left\|
   u^4\nabla u\right\|_{L^{3/2}(I;L^{6/5})} ,
\end{equation*}
where we have considered the specific admissible pairs $(4,4)$ and
$(3,6)$ for the cubic and quintic nonlinearities, respectively. Recall
that we already know that $u\in L^\infty(\R;H^1(\R^2))$, so the first
term on the right hand side is bounded uniformly in time. Write
\begin{equation*}
  \frac{3}{4} = \frac{2}{4}+\frac{1}{4},\quad
  \frac{2}{3} = \frac{4}{12}+\frac{1}{3}, \quad
 \frac{5}{6} = \frac{4}{6}+\frac{1}{6},
\end{equation*}
in which case, H\"older's inequality yields
\begin{align*}
 \|\nabla u\|_{S(I)}&\lesssim 1+ \left\|
   u\right\|_{L^{4}(I\times \R^2)}^2  \left\|
  \nabla  u\right\|_{L^{4}(I\times \R^2)}+ \left\|
   u\right\|_{L^{12}(I;L^{6})}^4 \left\| \nabla
                      u\right\|_{L^{3}(I;L^{6})} \\
  &\lesssim 1+ \left\|
   u\right\|_{L^{4}(I\times \R^2)}^2  \left\|
  \nabla  u\right\|_{S(I)}+ \left\|
   u\right\|_{L^{12}(I;L^{6})}^4 \left\| \nabla
                      u\right\|_{S(I)} .
\end{align*}
Recalling again that $u\in L^\infty(\R;H^1(\R^2))$,
\begin{equation*}
  \|u\|_{L^{12}(I;L^6)}^4\le
  \|u\|_{L^{\infty}(I;L^6)}^3\|u\|_{L^{3}(I;L^6)}\lesssim
  \|u\|_{L^{\infty}(I;H^1)}^3\|u\|_{S(I)}\le C \|u\|_{S(I)}.
\end{equation*}
Now since $\|u\|_{S(\R)}<\infty$, we can split $\R_t$ into finitely
many intervals on which the nonlinear terms in the above estimate can
be absorbed by the left hand side, so we conclude $\|\nabla
u\|_{S(\R)}<\infty$.
\smallbreak

Now if $u_0\in \Sigma$, we introduce the Galilean operator
\begin{equation*}
  J(t)=x+it\nabla,
\end{equation*}
which commutes with the free Schr\"odinger operator, i.e.
\begin{equation*}
  \big[J,i\d_t+\tfrac{1}{2}\Delta\big]=0.
\end{equation*}
Moreover,
\begin{equation}\label{eq:factor}
 J(t)u = it\, e^{i|x|^2/(2t)}\nabla\(u e^{-i|x|^2/(2t)}\),
\end{equation}
which implies that $J(t)u$ can be estimated like $\nabla u$ above. Thus 
$\|Ju\|_{S(\R)}<\infty$ and scattering in $\Sigma$ follows along the same lines as 
scattering in $H^1$ (see e.g. \cite{CazCourant}). 
\end{proof}

Next, in order to prove $\|u\|_{S(\R)}<\infty$ and thus \eqref{eq:dispcond}, we shall in the following distinguish between the case of subcritical mass, i.e. $\|u_0\|_{L^2}<\|Q\|_{L^2}$, and the critical case, 
where $\|u_0\|_{L^2}=\|Q\|_{L^2}$.


\subsection{Mass subcritical case} In this subsection, we suppose 
\begin{equation}\label{eq:subcrit}
  \|u_0\|_{L^2}^2= (1-\eta)\|Q\|_{L^2}^2,\quad \text{for some
  }0<\eta<1. 
\end{equation}
Since $u_0\in \Sigma$, we can rely on
the pseudo-conformal conservation law (derived initially in
\cite{GV79Scatt}, see also \cite{CazCourant}): 
\begin{equation}
  \label{eq:pseudoconf}
  \frac{d}{dt}\(\frac{1}{2}\|J(t) u\|_{L^2}^2
- \frac{ t^2}{2}\|u\|_{L^4}^{4} +\frac{ t^2}{3}\|u\|_{L^6}^{6} \) =
-\frac{2t}{3}\|u\|_{L^6}^6. 
\end{equation}
In view of \eqref{eq:factor}, we can rewrite 
\[
\|J(t) u\|_{L^2}^2 = t^2 \left\|\nabla\(u e^{-i|x|^2/(2t)}\)
\right\|_{L^2}^2 .
\]
The sharp Gagliardo--Nirenberg inequality \eqref{eq:GNsharp} when applied
to $u(t,x) e^{-i|x|^2/(2t)}$, then yields, together with \eqref{eq:subcrit}, that
\begin{align*}
  \|J(t) u\|_{L^2}^2
-  t^2\|u\|_{L^4}^{4} \ge &\, \|J(t)
u\|_{L^2}^2 -(1-\eta)\|J(t)u\|_{L^2}^2 \\
= &\, \eta \|J(t)
u\|_{L^2}^2.
\end{align*}
Hence, the pseudo-conformal conservation law implies
\begin{equation*}
 J(t) u\in L^\infty(\R_t;L^2(\R^2)).
\end{equation*}
Invoking \eqref{eq:factor} and general  Gagliardo--Nirenberg
inequalities, for $2\le r<\infty$,
\begin{equation}\label{eq:GNJ}
  \|u(t)\|_{L^r(\R^2)}\lesssim \|u(t)\|_{L^2(\R^2)}^{1-\theta}
  \(\frac{1}{t} \|J(t) u\|_{L^2}\)^{\theta},\quad \theta=1-\frac{2}{r},
\end{equation}
we infer $u\in L^q(\R;L^r(\R^2))$ for all admissible pairs,
i.e. $\|u\|_{S(\R)}<\infty$. 
\smallbreak

It was proved very recently in \cite{Cheng-p}, that the assumption $u_0\in \Sigma$ can
be relaxed to $u_0\in H^1(\R^2)$, by combining the Kenig--Merle roadmap
(\cite{KeMe06}, see also \cite{RaphaelBBK})
with a profile decomposition in
$L^2$. Such a decomposition arises naturally, even though the nonlinearity is not homogeneous. Indeed,
if the solution $u$ is dispersive, with fixed $L^2(\R^2)$ norm, we
expect it to behave like
\begin{equation*}
  |u(t,x)|\Eq t \infty\frac{1}{\lambda (t)}U\(\frac{x}{\lambda(t)}\),
\end{equation*}
with some fixed profile $U$ and some scaling factor $\lambda(t)\to
\infty$ as $t\to \infty$. In the case of 
scattering (i.e., asymptotically linear behavior), we have $\lambda (t)\approx
t$. Equivalence, however, is difficult to establish. In general,
for $\lambda(t)\to \infty$, we can only infer that the quintic
nonlinearity for large times $t$ becomes negligible when compared to the cubic one. The large time stability estimate in \cite{Cheng-p} relies
precisely on the solution to
\begin{equation}\label{eq:cubic}
  i\d_t v +\frac{1}{2}\Delta v=-|v|^2v,
\end{equation}
considered as a reference solution. Since this equation is
$L^2$-critical, we can follow \cite{MerleVega98} and expect a profile decomposition to arise at the $L^2$-level
(as opposed to $\dot H^s$-level for $s>0$, see, 
e.g. \cite{DuHoRo08,HoRo08,KeMe06}). The
solution to \eqref{eq:cubic} is guaranteed to be global in time and asymptotically linear if
$\|v\|_{L^2}<\|Q\|_{L^2}$. We emphasize that this property, 
conjectured for a long time, is actually hard to prove (see \cite{Dodson15}, and
\cite{PlanchonBBK} for the historical perspective),
and, hence, moving from the assumption $u_0\in \Sigma$ to $u_0\in H^1(\R^2)$
requires lots of technicalities. 

In the case
$\|u_0\|_{L^2}=\|Q\|_{L^2}$, the solution to the focusing cubic
Schr\"odinger equation may develop singularities, so the 
approach of \cite{Cheng-p} seems doomed. Supposing that we know that $u$ is dispersive, a critical aspect in our analysis below is to prove that
$\lambda(t)\approx t$.


\subsection{Mass critical case}\label{sec:mcrit}

Let now $u_0\in \Sigma$ with $\|u_0\|_{L^2}=\|Q\|_{L^2}$. The
pseudo-conformal conservation law yields 
\begin{equation}\label{eq:integratedpseudoconf}
   \underbrace{\frac{1}{2}\|J(t) u\|_{L^2}^2
- \frac{ t^2}{2}\|u\|_{L^4}^{4} }_{\ge 0\text{ from
\eqref{eq:GNsharp} and \eqref{eq:factor}}}+\frac{ t^2}{3}\|u\|_{L^6}^{6} 
=\frac{1}{2}\|x u_0\|_{L^2}^2
-\int_0^t\frac{2s}{3}\|u(s, \cdot)\|_{L^6}^6ds,
\end{equation}
hence
\begin{equation}\label{eq:L6control}
  \|u(t, \cdot)\|_{L^6}^{6} \lesssim \frac{1}{1+t^2}\ \text{and} \   \int_0^\infty t
  \|u(t, \cdot )\|_{L^6}^6dt<\infty.
\end{equation}
This does not rule out a behavior of the form
\begin{equation*}
   \|u(t, \cdot)\|_{L^6}^{6} \approx \frac{1}{t^2(\log t)^2}\quad \text{as
   }t\to \infty, 
 \end{equation*}
in which case $u\not \in L^3_tL^6_x$ (recall that $(3,6)$ is an admissible
pair). In other words, a direct use of the pseudo-conformal conservation law seems hopeless in the mass critical case, 
since we cannot access a convenient bound on $\|J(t)u\|_{L^2}$. 
In fact, we do not even have a moderate growth of this quantity, like
 $\O(t^\gamma)$ for some $\gamma<1/2$, as was
 exploited in \cite{TsutsumiYajima84}. 

 \begin{remark}
   Note that if the defocusing nonlinearity was weaker,
for instance quartic,
\begin{equation*}
  i\d_t u +\frac{1}{2}\Delta u =-|u|^2u+|u|^3u,\quad x\in \R^2,
\end{equation*}
then the same approach as above would yield
\begin{equation*}
   \|u(t, \cdot)\|_{L^5}^{5} \lesssim \frac{1}{1+t^2},
 \end{equation*}
 and so $u\in L^{10/3}_t L^5_x$, an admissible pair. We could then
 proceed as in the subcritical mass case. 
 \end{remark}
 We emphasize that in view of \eqref{eq:GNJ}, it suffices to show
 $Ju\in L^\infty(\R_t;L^2(\R^2))$, which is actually a stronger property
 than $\|u\|_{S(\R)}<\infty$. Suppose on the contrary that
 \begin{equation}
   \label{eq:JDV}
   \|J(t_n)u\|_{L^2(\R^2)}\Tend n \infty \infty \quad \text{for some
   }t_n\to \infty. 
 \end{equation}
 Consider $\psi$ given by
 \begin{equation}
   \label{eq:psi}
   \psi(t,x)=
   \frac{1}{t}u\(\frac{-1}{t},\frac{x}{t}\)e^{i|x|^2/(2t)},\quad t\not =0.
 \end{equation}
 As is well-known, this transform exchanges the large time and finite
 time r\'egimes. In the case of the $L^2$-critical nonlinear
 Schr\"odinger equation, $u$ and $\psi$ solve the same equation. In the
 present case, $\psi$ solves the non-autonomous equation
 \begin{equation*}
   i\d_t \psi +\frac{1}{2}\Delta \psi = -|\psi|^2\psi+t^2|\psi|^4\psi.
 \end{equation*}
 In view of \eqref{eq:factor} and \eqref{eq:JDV},
 \begin{equation}\label{eq:gradpsi}
   \|\nabla\psi(\tau_n, \cdot)\|_{L^2(\R^2)} =
   \left\|J\(\frac{-1}{\tau_n}\)u\right\|_{L^2(\R^2)} \Tend
   n\infty\infty,\quad\tau_n:=\frac{-1}{t_n}\Tend n
   \infty 0^-.
 \end{equation}
We obviously have $\|\psi(t)\|_{L^2}=\|u_0\|_{L^2}=\|Q\|_{L^2}$. We
use this property to rely on fine rigidity results associated to the
$L^2$-critical nonlinear Schr\"odinger equation (with critical mass),
as first proved in \cite{Merle93}, and revisited in
\cite{HmidiKeraaniIMRN}. More precisely, as will be proven 
below, \eqref{eq:JDV} implies
\begin{equation*}
  \rho_ne^{i\theta_n} \psi(\tau_n,\rho_nx+x_n) \Tend
  n\infty Q(x)\text{
    in }H^1(\R^2), 
\end{equation*}
for some $\rho_n,\theta_n\in \R$ and $x_n\in \R^2$. The crucial aspect is
that this convergence is strong in $H^1(\R^2)$. We then show that
$|\rho_n|\lesssim 
|\tau_n|$. This behavior rules out the first part of
\eqref{eq:L6control}, and hence
\eqref{eq:JDV} cannot hold.
\smallbreak

\noindent {\bf Step 1. A-priori estimates.} The idea of this proof by contradiction is to rely on fine properties
established for mass-critical, blowing-up solutions in the case of
$L^2$-critical nonlinear Schr\"odinger equations. We start with the fact that, in view of
\eqref{eq:integratedpseudoconf}, 
\begin{equation*}
  0\le \|J(t)u\|_{L^2}^2 -
  t^2\|u(t)\|_{L^4}^4\le \|x u_0\|_{L^2}^2. 
\end{equation*}
Hence, in terms of $\psi$, defined via \eqref{eq:psi}, we have
\begin{equation}\label{eq:Epsi}
  0\le \|\nabla \psi(t)\|_{L^2}^2 -
  \|\psi(t)\|_{L^4}^4\le \|x u_0\|_{L^2}^2.
\end{equation}
Next, we recall standard estimates on the virial quantity: The identity
\begin{equation*}
  \frac{d}{dt}\int_{\R^2}|x|^2|u(t,x)|^2dx = 2\IM\int_{\R^2}\bar
  u(t,x)x\cdot \nabla u(t,x)dx,
\end{equation*}
together with Cauchy-Schwarz, and the boundedness $u\in
L^\infty(\R;H^1(\R^2))$, yields
\begin{equation*}
  \int_{\R^2}|x|^2|u(t,x)|^2dx \lesssim 1+|t|,\quad \forall t\in \R.
\end{equation*}
In turn, \eqref{eq:psi} implies
\begin{equation}\label{eq:apviriel}
   \int_{\R^2}|x|^2|\psi(t,x)|^2dx \lesssim 1,\quad \forall t\in[-1,0[.
\end{equation}
Recall that the (conserved) energy
associated to the focusing cubic Schr\"odinger equation
\begin{equation*}
  i\d_t v +\frac{1}{2}\Delta v = -| v |^2 v,
\end{equation*}
is
\begin{equation*}
  E_{\rm cub} (v)=\|\nabla
  v \|_{L^2}^2-\| v\|_{L^4}^4. 
\end{equation*}
This will be used to prove a-priori
estimates on truncated versions of the virial quantity. Following \cite{HmidiKeraaniIMRN}, we introduce a nonnegative
radial $\chi \in C_0^\infty(\R^2)$, such that
\begin{equation*}
  \chi(x) =|x|^2\quad\text{if }|x|<1,\quad |\nabla \chi(x)|^2\lesssim \chi(x),
\end{equation*}
and for $\underline x\in \R^2$ and every $p\in \N^*$, we define
\begin{equation*}
  \chi_{p,\underline x}(x):=p^2\chi\(\frac{x-\underline x}{p}\),\quad g_{p,\underline
    x}(t):=\int_{\R^2}\chi_{p,\underline x}(x)|\psi(t,x)|^2dx.
\end{equation*}
The function $g_{p,\underline  x}$ is a truncated virial, centered at
point $\underline x\in \R^2$. A computation shows that
\begin{equation}\label{eq:gdot}
  \dot g_{p,\underline  x}(t) = 2\IM\int_{\R^2} \bar
  \psi(t,x)\nabla \chi_{p,\underline x}(x)\cdot\nabla \psi(t,x)dx. 
\end{equation}
We now invoke an argument introduced by Valeria Banica
in \cite{Banica2004}: The fact that $\|\psi\|_{L^2}=\|Q\|_{L^2}$
and \eqref{eq:GNsharp} imply that for every $t\in [-1,0[$,
\begin{equation*}
  E_{\rm cub}\( e^{is\chi_{p,\underline x}} \psi(t)\)\ge 0,\quad \forall s\in \R.
\end{equation*}
The above quantity is a polynomial of degree two in $s$,
\begin{align*}
   E_{\rm cub}\( e^{is\chi_{p,\underline x} }\psi(t)\)&= E_{\rm cub}\(
   \psi(t)\)-2s\IM \int_{\R^2} \bar \psi(t,x)\nabla
   \chi_{p,\underline x}(x)\cdot\nabla \psi(t,x)dx \\
& \quad   + s^2\int_{\R^2}|\nabla \chi_{p,\underline x}(x)|^2|\psi(t,x)|^2dx,
 \end{align*}
 so its discriminant is non-positive, i.e.
 \begin{align*}
   \left|2\IM \int_{\R^2} \bar \psi(t,x)\nabla
   \chi_{p,\underline x}(x)\cdot\nabla \psi(t,x)dx \right|^2&\le 4  E_{\rm cub}\(
   \psi(t)\) \int_{\R^2}|\nabla \chi_{p,\underline x}(x)|^2|\psi(t,x)|^2dx\\
   &\lesssim \|x u_0\|_{L^2}^2g_{p,\underline  x}(t) .
 \end{align*}
Here we have used \eqref{eq:Epsi} and the definition of
$\chi$. Together with \eqref{eq:gdot} this yields
\begin{equation*}
  \left|\dot g_{p,\underline  x}(t) \right|\lesssim \sqrt{g_{p,\underline  x}(t) },
\end{equation*}
and, hence, by integration,
\begin{equation}\label{eq:decrvirie}
  \left|\sqrt{g_{p,\underline  x}(t) }-\sqrt{g_{p,\underline  x}(\tau)
    }\right|\lesssim |t-\tau|,\quad -1\le t\le\tau<0.
\end{equation}

\noindent {\bf Step 2. Blow-up profile along the sequence $\tau_n$.} 
Following \cite{HmidiKeraaniIMRN}, we set
\begin{equation}\label{eq:defrho}
  \rho_n =\frac{\|\nabla Q\|_{L^2}}{\|\nabla \psi(\tau_n, \cdot)\|_{L^2}},\quad
  v_n(x) :=\rho_n \psi\(\tau_n,\rho_nx\).
\end{equation}
We have $\|v_n\|_{L^2}=\|Q\|_{L^2}$, $\|\nabla v_n\|_{L^2}=\|\nabla
Q\|_{L^2}$ for all $n\in \N$. In view of \eqref{eq:gradpsi},
\begin{equation*}
  E_{\rm cub}\(v_n\)=\rho_n^2 E_{\rm
    cub}\(\psi(\tau_n)\)\Tend n\infty 0, 
\end{equation*}
since $ E_{\rm cub}\(\psi(t)\)$ is bounded from \eqref{eq:Epsi}.
We have all the ingredients necessary to invoke Theorem~2.3 from
\cite{HmidiKeraaniIMRN}, \emph{along the sequence $\tau_n$}: there exist 
$x_n\in \R^2$ and $\theta_n\in \R$ such that
\begin{equation}\label{eq:CVsequ}
  \rho_n e^{i\theta_n}\psi\(\tau_n,\rho_nx+x_n\)\Tend
  n\infty 
  Q,\quad\text{strongly in }H^1(\R^2). 
\end{equation}
\smallbreak

\noindent {\bf Step 3. Lower bound on $\|\nabla
  \psi(\tau_n,\cdot)\|_{L^2}$.} The convergence \eqref{eq:CVsequ}
implies that, in the sense of measures,
\begin{equation*}
  |\psi(\tau_n,x)|^2dx-Q(x)\delta_{x=x_n}\rightharpoonup 0.
\end{equation*}
This and the bound \eqref{eq:apviriel} imply that the sequence 
$(x_n)_{n\in \N}$ is bounded, hence we can extract a subsequence $(x_{n'})_{n'\in \N}$ such that
\begin{equation*}
  x_{n'}\to \underline x,\quad \underline x\in \R^2.
\end{equation*}
Using $\chi_{p,\underline x}$ as a test function, the above
convergence implies that $g_{p,\underline x}(\tau_{n'})\to
0$, as $\chi_p(0)=0$, and thus, \eqref{eq:decrvirie} yields
\begin{equation*}
  \left|\sqrt{g_{p,\underline  x}(t) }\right|\lesssim |t|,\quad -1\le t<0.
\end{equation*}
Now letting $p\to \infty$, Fatou's lemma implies
\begin{equation*}
  \int_{\R^2}|x-\underline x|^2|\psi(t,x)|^2dx\lesssim t^2.
\end{equation*}
Using of the classical uncertainty principle, i.e.
\begin{equation*}
  \|\psi(t,\cdot)\|_{L^2(\R^2)}^2\le \(\int_{\R^2}|x-\underline
  x|^2|\psi(t,x)|^2dx\)^{1/2}\|\nabla \psi(t,\cdot)\|_{L^2(\R^2)},
\end{equation*}
allows us to conclude
\begin{equation*}
  \|\nabla \psi(t,\cdot)\|_{L^2(\R^2)}\gtrsim \frac{1}{t}.
\end{equation*}
In particular, along the sequence $\tau_n$, we find
\begin{equation*}
0<  \rho_n\lesssim \tau_n.
\end{equation*}

\noindent {\bf Step 4. Conclusion.}
In view of \eqref{eq:CVsequ}, and since 
$H^1(\R^2)\hookrightarrow L^6(\R^2)$,
\begin{equation*}
  \|\psi(\tau_n, \cdot)\|_{L^6}^6 \Eq n\infty
  \frac{1}{\rho_n^4}\|Q\|_{L^6}^6 \gtrsim \frac{1}{\tau_n^4}.
\end{equation*}
Recalling that $\psi$ and $u$ are linked via \eqref{eq:psi}, we find
  \begin{equation*}
    \|u(t_n, \cdot)\|_{L^6}^6 =
    \frac{1}{t_n^4}\left\|\psi\(\frac{-1}{t_n},
      \cdot\)\right\|_{L^6}^6 \gtrsim 1,
  \end{equation*}
  which is incompatible with the first part of \eqref{eq:L6control}. Therefore,
   \eqref{eq:JDV} cannot hold, i.e. $Ju\in  L^\infty(\R_t;L^2(\R^2))$, and hence
  Theorem~\ref{theo:2Ddispersion} follows from Lemma~\ref{lem:reduc}. 

\section{Existence of solitons and first properties}
\label{sec:soliton}

\subsection{A priori estimates}

Suppose we have a solution $u(t,x) =e^{i\omega t}\phi(x)$, with $\phi$
sufficiently smooth and localized. Then \eqref{eq:nls} becomes
\begin{equation}
  \label{eq:soliton}
  -\frac{1}{2}\Delta \phi -|\phi|^2\phi +|\phi|^4\phi +\omega \phi=0.
\end{equation}
\begin{proposition}[A priori estimates for solitary waves]\label{prop:solitonapriori}
  Let $1\le d\le 3$. If $\phi\in H^1(\R^d)$ solves \eqref{eq:soliton}, then we have:
  \begin{enumerate}
  \item Pohozaev identities:
    \begin{equation}
  \label{eq:phi1}
  \frac{1}{2}\int_{\R^d}|\nabla \phi|^2 \, dx
  -\int_{\R^d}|\phi|^4 \, dx+\int_{\R^d}|\phi|^6 \, dx+\omega \int_{\R^d}|\phi|^2  \, dx= 0,
\end{equation}
\begin{equation}
  \label{eq:phi2}
  \frac{d-2}{2} \int_{\R^d}|\nabla \phi|^2 \, dx
 -\frac{d}{2}\int_{\R^d}|\phi|^4 \, dx+\frac{d}{3}\int_{\R^d}|\phi|^6 \, dx+\omega
 d \int_{\R^d}|\phi|^2  \, dx= 0.
\end{equation}
\item If  $\phi\not \equiv 0$, then
  $0<\omega<\tfrac{3}{16}$.
\item If $d=2$ and $\phi\not=0$, then $\|\phi\|_{L^2}>\|Q\|_{L^2}$,
  where $Q$ is cubic ground state solution to \eqref{eq:Q2D}. 
\item If in addition $\phi\in L^\infty\cap C^2$ is real-valued, then
  \begin{equation}\label{eq:phiLinfini}
    \|\phi\|_{L^\infty(\R^d)}\le \sqrt{\frac{1+\sqrt{1-4\omega}}{2}}.
  \end{equation}
 \end{enumerate}
\end{proposition}
\begin{proof}
For item (1), we quickly recall the method to derive Pohozaev identities formally,
and refer to \cite{BL83a} for a rigorous justification via density type arguments.
Firstly, multiplying \eqref{eq:soliton} by $\bar \phi$ and integrating yields
\eqref{eq:phi1}. In particular, we infer $\omega\in \R$. Secondly, by multiplying \eqref{eq:soliton} with 
$x\cdot \nabla \bar \phi$ and
integrating by parts we obtain \eqref{eq:phi2}.
For $d=2$, subtracting \eqref{eq:phi2} from \eqref{eq:phi1} yields
\begin{equation*}
  \frac{1}{2}\int_{\R^2}|\nabla \phi(x)|^2 \, dx
  +\frac{1}{3}\int_{\R^2}|\phi(x)|^6 \, dx=\omega \int_{\R^2}|\phi(x)|^2  \, dx,
\end{equation*}
hence $\omega>0$ unless $\phi\equiv 0$.
In the case $d=3$, we obtain similarly
\begin{equation*}
  \frac{1}{2}\int_{\R^3}|\phi(x)|^4  \, dx=2\omega \int_{\R^3}|\phi(x)|^2  \, dx,
\end{equation*}
and thus we arrive at the same conclusion.
\smallbreak

(2) From now on, we shall denote
\[ F(s) = \frac{1}{4}s^4-\frac{1}{6}s^6\]
and set
\begin{equation*}
  \omega^* = \sup\left\{\omega>0;\quad \frac{\omega}{2}s^2-F(s)<0\text{ for
    some }s>0\right\}.
\end{equation*}
A computation then shows $\omega^* =\tfrac{3}{16}$. In particular, if $\omega\ge \tfrac{3}{16}$, we have the pointwise relation
\[-\frac{1}{2}|\phi(x)|^4+\frac{1}{3}|\phi(x)|^6+\omega
|\phi(x)|^2\ge 0 ,\quad \forall x\in \R^d,\]
which, in view \eqref{eq:phi2}, implies $\phi\equiv 0$ for $d\ge
2$. In the case $d=1$, the conclusion follows from ODE arguments, and
more precisely \cite[Theorem~5]{BL83a}. 
\smallbreak

(3) We suppose $d=2$ and introduce
\begin{equation*}
  \gamma = \frac{\|\phi\|_{L^4}^4}{\|\nabla \phi\|_{L^2}^2},
\end{equation*}
which allows us to rewrite \eqref{eq:phi1} as
\[
\(\frac{1}{2}-\gamma\)\int_{\R^2}|\nabla \phi|^2
  +\int_{\R^2}\phi^6+\omega\int_{\R^2}\phi^2=0,
\]
Similarly, we can rewrite \eqref{eq:phi2} for $d=2$, by using $\gamma$, in the following form
\begin{align*}
  \frac{\gamma}{2}\int_{\R^2}|\nabla \phi|^2-\frac{1}{3}\int_{\R^2}|\phi|^6-\omega
  \int_{\R^2}|\phi|^2 = 0.
\end{align*}
Combining these identities, we infer
\begin{equation*}
  \int_{\R^2}|\phi|^6 =\frac{3(\gamma-1)}{4} \int_{\R^2}|\nabla \phi|^2,
\end{equation*}
and in particular $\gamma>1$, i.e. $ \|\phi\|_{L^4}^4 > \|\nabla \phi\|_{L^2}^2$. 
In view of the sharp Gagliardo-Nirenberg inequality \eqref{eq:GNsharp}, this consequently implies that the mass of the cubic-quintic ground states
satisfies $\|\phi\|_{L^2}>\|Q\|_{L^2}$.
\smallbreak

(4) Let $\phi\in C^2$ be a real-valued bounded solution to
\eqref{eq:soliton}. Suppose that $\phi$ reaches its 
  maximum at $x_0\in \R^d$. Then $\Delta \phi(x_0)\le 0$, and hence
  \begin{equation*}
   ( -\phi^3+\phi^5 +\omega\phi)_{\mid x=x_0}\le 0.
  \end{equation*}
  Writing
  \begin{equation*}
    \omega\phi -\phi^3+\phi^5 =\phi
    \(\phi^2-\frac{1-\sqrt{1-4\omega}}{2}\)\(\phi^2-\frac{1+\sqrt{1-4\omega}}{2}\), 
  \end{equation*}
  we see that
  \begin{equation*}
    \phi(x_0)\le \sqrt{\frac{1+\sqrt{1-4\omega}}{2}}.
  \end{equation*}
  Reasoning similarly for a minimum of $\phi$, we infer \eqref{eq:phiLinfini}.
\end{proof}


\subsection{Existence and uniqueness}

Denote $f(s) =s^3-s^5$ and
\[ F(s) =\int_0^s f(\tau) \, d\tau = \frac{1}{4}s^4-\frac{1}{6}s^6, 
\]
as before. We have already seen that 
\begin{equation*}
   \sup\left\{\omega>0;\quad \frac{\omega}{2}s^2-F(s)<0\text{ for
    some }s>0\right\} =\frac{3}{16} .
\end{equation*}
Then according to \cite{BL83a} (treating the case $d=1$ or $d=3$) and \cite{BGK83}
(treating the case $d=2$), for all $\omega\in ]0,\omega^*[$, there exists a solution
$\phi_\omega$ of \eqref{eq:soliton}. Uniqueness of $\phi_\omega$ in
$d=1$  is proven in
\cite{BL83a}, while in $d=3$ uniqueness follows from \cite{SerrinTang00}, as pointed out in \cite{KOPV17}.
Finally, for $d=2$, we infer uniqueness from the results of \cite{Jang10}, where we emphasize that the 
assumptions made there correspond more closely to those made to prove
existence.
\smallbreak

We recall that the action, defined in the introduction, is given by
\begin{equation*}
  S(\phi) = \frac{1}{2}\|\nabla\phi\|_{L^2}^2 +\omega\|\phi\|_{L^2}^2
  -2V(\phi),\quad\text{where}\quad V(\phi)
  = \int_{\R^d}F\(\phi(x)\)dx,
\end{equation*}
and satisfies
\begin{equation}\label{eq:action-energy}
  S(\phi)=E(\phi) + \omega\|\phi\|_{L^2}^2=E(\phi)+\omega M(\phi).
\end{equation}
As established in \cite[Lemma~2.3]{CiJeSe09} (which requires the
nonlinearity to be energy-subcritical, i.e. $d\le 2$ in our
case), every minimizer of the action is of the 
form
\begin{equation*}
  \varphi(x) = e^{i\theta} \phi(x),
\end{equation*}
for some constant $\theta\in \R$, and where $\phi$ is a positive least
action solution of 
\eqref{eq:soliton}. Then \cite[Proposition~4]{ByJeMa09} implies that
\begin{equation*}
  \phi(x)=\phi_\omega(x-x_0),
\end{equation*}
for some $x_0\in \R^d$, and we recall that $\phi_\omega$ is the unique
radial, positive minimizer of the action. The same is true for $d=3$
(where the defocusing nonlinearity is energy-critical), as explained
in \cite{MorozMuratov2014}. 
We summarize all of these results in the proposition below.

\begin{proposition}[Existence and uniqueness of ground states] \label{prop:phi}
 Let $1\le d\le 3$.  Suppose that
  \begin{equation*}
  0<\omega<\frac{3}{16}.
\end{equation*}
Then \eqref{eq:soliton} has a unique radial, real-valued solution $\phi_\omega$ such that
\begin{enumerate}
\item $\phi_\omega>0$ on $\R^d$.
  \item $\phi_\omega$ is radially symmetric, $\phi_\omega(x)=\phi(r)$, where
    $r=|x|$, and $\phi$ is a non-increasing function of $r$.
  \item $\phi_\omega\in C^2(\R^d)$.
    \item The derivatives of order at most two of $\phi_\omega$ decay
      exponentially:
      \begin{equation*}
        \exists \delta>0,\quad |\d^\alpha \phi_\omega(x)|\lesssim
        e^{-\delta|x|},\quad |\alpha|\le 2.
      \end{equation*}
    \item For every solution $\varphi$ to \eqref{eq:soliton},
      \begin{equation*}
        0< S(\phi_\omega)\le S(\varphi).
      \end{equation*}
    \item Every minimizer $\varphi$ of the action $S(\phi)$ is of the form
\begin{equation*}
  \varphi(x) = e^{i\theta} \phi_\omega(x-x_0),
\end{equation*}
for some constants $\theta\in \R$, $x_0\in \R^d$.  
    \end{enumerate}
\end{proposition}


\section{Further properties of the ground states}
\label{sec:further}

\subsection{Asymptotic r\'egimes for limiting values of $\omega$}
\label{sec:asympt-omega}

The results below, regarding the asymptotic r\'egimes $\omega\to 0$ and $\omega\to
\tfrac{3}{16}$, are included in \cite{KOPV17} or
\cite{LewinRotaNodari-p}.

\begin{proposition}[Asymptotics of the ground state mass]\label{prop:omega-asym}
  Let $d=2$ or $3$ and $Q$ denote the cubic nonlinear ground state. The
map $\omega\mapsto \phi_\omega$ given by Proposition~\ref{prop:phi} is
real analytic and admits the following asymptotic behavior:
\begin{enumerate} 
\item In the limit $\omega\to 0$, we have:
\begin{itemize}
  \item[(a)] If $d=2$,
    \begin{equation*}
      M(\phi_\omega) = M(Q) + \frac{2\omega}{3}\|Q\|_{L^6}^6+\O\(\omega^2\),
    \end{equation*}
  \item[(b)] If $d=3$,
      \begin{equation*}
      M(\phi_\omega) = \frac{1}{\sqrt\omega}M(Q) + \frac{\sqrt\omega}{2}\|Q\|_{L^6}^6+\O\(\omega^{3/2}\).
    \end{equation*}
  \end{itemize}
  \item  In the limit $\omega\to \tfrac{3}{16}$, it holds
  \begin{equation*}
    M(\phi_\omega) + \frac{\d M(\phi_\omega)}{\d\omega}\Tend \omega {\tfrac{3}{16}} \infty.
  \end{equation*}
\end{enumerate}
\end{proposition}
To turn the singular limit $\omega\to 0$ 
into a regular limit, one proceeds as in \cite{MorozMuratov2014,KOPV17}
and changes the unknown function $\phi_\omega$ into
\begin{equation*}
  \psi_\omega(x)= \frac{1}{\sqrt\omega}\phi_\omega\(\frac{x}{\sqrt\omega}\).
\end{equation*}
Then \eqref{eq:soliton} is equivalent to
\begin{equation*}
  -\frac{1}{2}\Delta \psi_\omega+\psi_\omega-\psi_\omega^3+\omega \psi_\omega^5=0.
\end{equation*}
We thereby note that the rescaling $\phi_\omega\mapsto \phi_\omega$ is
$L^2$-unitary exactly for $d=2$. As before, we denote by $Q$ 
the unique non-negative, radially symmetric ground state solution to
\begin{equation*}
  -\frac{1}{2}\Delta  Q+Q-Q^3=0,\quad x\in \R^d,
\end{equation*}
and consider the linearized operator
\[L: f \mapsto -\frac{1}{2}f-3Q^2f+f.\]
Then $L: H^1_{\rm rad}\to H^{-1}_{\rm rad}$ is an isomorphism, where $H^1_{\rm rad}$ denotes the Sobolev space of radial $H^1$ functions. 
Invoking the implicit function theorem, as well as uniqueness for
\eqref{eq:soliton}, we have, in $H^1(\R^d)$ and as $\omega\to 0$,
\begin{equation}\label{eq:psiom}
  \psi_\omega(x) = Q(x) -\omega\(L^{-1}Q^5\)(x)+\O(\omega^2),
\end{equation}

In particular, in the case $d=2$, we infer
\begin{equation*}
  \|\phi_\omega\|_{L^2(\R^2)} = \|\psi_\omega\|_{L^2(\R^2)} \Tend
  \omega 0 \|Q \|_{L^2(\R^2)},
\end{equation*}
thus showing that ground states for the cubic-quintic NLS in
2D have mass strictly larger but arbitrarily close to that of the cubic ground
state $Q$. Noting more precisely the relation
(\cite[Proposition~B.1]{Weinstein85}) 
\begin{equation*}
  L\(Q+x\cdot \nabla Q\)=-2Q,
\end{equation*}
and recalling \eqref{eq:psiom}, we infer, 
\begin{align*}
  M(\psi_\omega) &=
 \<Q-\omega\(L^{-1}Q^5\),Q-\omega\(L^{-1}Q^5\)\> +\O(\omega^2) \\
  &= M(Q)-2 \omega\<Q,L^{-1}Q^5\>+\O(\omega^2) \\
   &= M(Q)+ \omega\< L\(Q+x\cdot \nabla Q\),L^{-1}Q^5\>+\O(\omega^2) \\
 &= M(Q)+ \omega\< Q+x\cdot \nabla Q,Q^5\>+\O(\omega^2) \\
  &= M(Q)+ \omega \(1-\frac{d}{6}\)\|Q\|_{L^6(\R^d)}^6 +\O(\omega^2) .
\end{align*}
Translating this in terms of $\phi_\omega$ yields the first part
of the proposition. 
\smallbreak

The second part of the proposition is proven in
\cite{LewinRotaNodari-p}. We also note that for $d=3$, in \cite[Theorem~2.2,
(v)]{KOPV17} the authors prove that 
\begin{equation*}
 \(\tfrac{3}{16}-\omega\)^{-3}\lesssim  M(\phi_\omega)\lesssim \(\tfrac{3}{16}-\omega\)^{-3},
\end{equation*}
which, however, does not rule out possible oscillations of $M(\phi_\omega)$.

The information provided by Proposition~\ref{prop:omega-asym} is
interesting in view of Grillakis-Shatah-Strauss theory: for $d=2$,
$\omega\mapsto M(\phi_\omega)$ is increasing near $\omega=0$ and near
$\omega=\tfrac{3}{16}$, while for $d=3$, this map is decreasing near $\omega=0$
and increasing near $\omega=\tfrac{3}{16}$. As mentioned in the introduction,
we actually expect this map to be increasing on the full interval
$]0,\tfrac{3}{16}[$ when $d=2$, 
decreasing on $]0,\omega_0[$ and increasing on $]\omega_0,\tfrac{3}{16}[$ when
$d=3$. 


\subsection{Spectral properties}
\label{sec:spectral-properties}

To take advantage of the above properties, we
have to check the spectral Assumption~3 imposed in \cite{GSS87}.
To state the spectral assumption, we write the
second order derivative of the action as
\begin{equation*}
  \<S''(\phi_\omega)w,w\> = \frac{1}{2}\<L_1 u,u\>+\frac{1}{2}\<L_2 v,v\>,
\end{equation*}
where $w=u+iv$. In our case, we have
\begin{align*}
  L_1 &= -\frac{1}{2}\Delta +\omega -3\phi_\omega^2 +5\phi_\omega^4 ,\\
  L_2 & =-\frac{1}{2}\Delta +\omega -\phi_\omega^2 +\phi_\omega^4  .
\end{align*}
We then need to check:
\begin{assumption}\label{hyp:spectral}
  For each $\omega\in ]0,\tfrac{3}{16}[$, the Hessian $S''(\phi_\omega)$ has exactly one
  negative eigenvalue; its kernel is
  spanned by $i\phi_\omega$ and $\nabla \phi_\omega$, and the rest of its spectrum is
  positive and bounded away from zero. 
\end{assumption}
If this holds true, then:
\begin{itemize}
\item[(a)] If $\frac{\d}{\d\omega}M(\phi_\omega)>0$,
  then the standing wave $e^{i\omega t}\phi_\omega(x)$ is orbitally
  stable.
  \item[(b)] If $ \frac{\d}{\d\omega}M(\phi_\omega)<0$,
  then the standing wave $e^{i\omega t}\phi_\omega(x)$ is
  unstable.
\end{itemize}

Indeed, the authors of \cite{KOPV17} proved that
Assumption~\ref{hyp:spectral} holds true for the cubic-quintic NLS in
3D. These properties are established in
\cite{LewinRotaNodari-p,LewinRotaNodari2015} 
in a more general setting, covering \eqref{eq:nls} for $d=2$. 

\begin{proposition}[Proposition~2.4 from \cite{KOPV17}]
  Fix $\ell=0,1,2,\dots$, and consider the restriction of $L_1$ to
  functions of the form $f(|x|)Y(x/|x|)$, where $Y$ is a spherical
  harmonic of degree $\ell$.
  \begin{enumerate}
  \item When $\ell=0$, the operator has exactly one negative
    eigenvalue; it is simple.
    \item When $\ell =1$, there are no negative eigenvalues. Zero is an
      eigenvalue and its eigenspace is spanned by the three components
      of $\nabla \phi_\omega$.
  \item When $\ell\ge 2$, the operator is positive definite. 
  \end{enumerate}
\end{proposition}
The proof of this result relies on Sturm Oscillation Theorem, since
the analysis boils down to second order ODEs for the radial function
$f$. Note that the proof from \cite{KOPV17} can be readily adapted to the
2D case, by replacing spherical harmonics with functions of the form
$e^{i\ell \theta}$ in radial coordinates. The above proposition is
complemented by the following one:
\begin{proposition}[Proposition~2.5 from \cite{KOPV17}, Lemma~3 from \cite{LewinRotaNodari2015}] Let $\delta=\delta(r)$
  be the solution to
  \begin{equation*}
    -\frac{1}{2}\delta'' -\frac{1}{r}\delta' +
    \(5\phi_\omega^4-3\phi_\omega^2+\omega\)\delta=0 
  \end{equation*}
  obeying $\delta(0)=1$. Then $\delta(r)\to -\infty$ as $r\to
  \infty$. Correspondingly, zero is not an eigenvalue of $L_1$
  restricted to radial functions. 
\end{proposition}
Thus, all conditions necessary to invoke the Grillakis-Shatah-Strauss theory are satisfied and one can infer (in-)stability of ground states 
from the properties of the map $\omega\mapsto M(\phi_\omega)$.
At this stage, all items of Theorem~\ref{theo:2Dstability} and Proposition~\ref{prop:soliton3D} are proved, except the 
ones concerning $\mathcal E(\rho)$-stability of solitary waves.


\section{Orbital stability of the set of energy minimizers}
\label{sec:cazenave-lions}

\subsection{Two-dimensional case}
\label{sec:2D}

We now prove the fourth point of Theorem~\ref{theo:2Dstability}.\\

\noindent {\bf First step.} We show that for all $\rho>M(Q)$,
\begin{equation*}
 \inf\left\{E(u)\, ;\, u\in \Gamma(\rho)\right\} =-\nu,
\end{equation*}
for some finite $\nu>0$. To prove that the infimum is finite, we use 
H\"older's inequality \eqref{eq:holder}, to infer
\begin{equation*}
  E(u) \ge \frac{1}{2}\|\nabla u(t)\|_{L^2(\R^d)}^2 -\frac{\sqrt\rho}{2}\|
  u(t)\|_{L^6(\R^d)}^3+ \frac{1}{3}\|  u(t)\|_{L^6(\R^d)}^6,
\end{equation*}
and thus $E(u)$ is bounded from below. To see that the infimum is negative,
consider the $L^2$-invariant  scaling, for $\lambda>0$,
\begin{equation*}
  u_\lambda(x)=\lambda^{d/2}u(\lambda x),
\end{equation*}
which, for $d=2$, implies
\begin{equation*}
  E(u_\lambda) =\frac{\lambda^2}{2}\(\|\nabla u\|_{L^2}^2
  -\|u\|_{L^4}^4 +\frac{2}{3}\lambda^2\|u\|_{L^6}^6\).
\end{equation*}
In view of the sharp Gagliardo-Nirenberg inequality, and since $\|u\|_{L^2}^2>\|Q\|_{L^2}^2$, we may choose a profile $u\in H^1$
so that the terms independent of
$\lambda$ inside the parentheses become negative, e.g., take 
\[
u = \sqrt{\frac{\rho}{M(Q)}} \, Q, \ \text{with $\lambda>0$
sufficiently small}.
\]

\noindent {\bf Second step.} Any minimizing sequence is bounded away
from zero in $L^4$. Let $(u_n)_{n\ge 0}$ be a minimizing sequence: for
$n$ sufficiently large, $E(u_n)\le -\nu/2$, hence
\begin{equation*}
  \|u_n\|_{L^4}^4\ge \nu>0.
\end{equation*}

\noindent {\bf Third step.} In view of \cite{PL284a} (see also
\cite[Proposition~1.7.6]{CazCourant}),  we have the standard
trichotomy of concentration compactness. From the second step,
vanishing is ruled out, so we have to rule out dichotomy to infer
compactness. Arguing by contradiction, suppose that, after extraction of suitable subsequences,
there exist $(v_k)_{k\ge 0}$, $(w_k)_{k\ge 0}$ in $H^1(\R^2)$, such
that
\begin{align*}
  & \operatorname{supp} v_k\cap \operatorname{supp} w_k=\emptyset,\  |v_k|+|w_k|\le |u_{n_k}|,\ \|v_k\|_{H^1} + \|w_k\|_{H^1} \le C\|u_{n_k}\|_{H^1} ,
\end{align*} 
satisfying
\begin{align*} 
  \|v_k\|_{L^2}^2\Tend k \infty \theta\rho,\quad
    \|w_k\|_{L^2}^2\Tend k \infty (1-\theta)\rho,\quad\text{for some
    }\theta\in ]0,1[,
  \end{align*}
and
\begin{align*} 
&\liminf_{k\to \infty} \(\int |\nabla u_{n_k}|^2 - \int|\nabla
  v_k|^2 -\int |\nabla  w_k|^2 \)\ge 0,\\
& \left|\int |u_{n_k}|^p -\int|v_k|^p- \int|w_k|^p\right|\Tend k \infty0,
  \end{align*}
for all $2\le p<\infty$. We infer
\begin{equation*}
  \liminf_{k\to \infty}\(E\(u_{n_k}\)-E(v_k)-E(w_k)\)\ge 0,
\end{equation*}
hence
\begin{equation}\label{eq:lower}
  \limsup_{k\to \infty} \(E(v_k)+E(w_k)\)\le -\nu.
\end{equation}
Following an idea from \cite{CoJeSq10}, we then use a scaling argument
rather than a multiplicative one as in \cite{CaLi82}. Let
\begin{align*}
  \tilde v_k(x) & = v_k\(\lambda_k^{-1/2}x\),\quad  \lambda_k =
                  \frac{\rho}{\|v_k\|_{L^2}^2} \\
  \tilde w_k(x)& = w_k\(\mu_k^{-1/2}x\),\quad \mu_k =
                  \frac{\rho}{\|w_k\|_{L^2}^2} .
\end{align*}
Since $ \tilde v_k$ and $ \tilde w_k$ have mass $\rho $,
\begin{equation*}
  E(\tilde v_k),\ E(\tilde w_k)\ge -\nu.
\end{equation*}
On the other hand, we compute
\begin{equation*}
  E(\tilde v_k) = \lambda_k\(\frac{1}{2\lambda_k}\int |\nabla v_k|^2
  -\frac{1}{2}\int |v_k|^4 +\frac{1}{3}\int |v_k|^6\),
\end{equation*}
and so
\begin{equation*}
  E(v_k) = \frac{1}{\lambda_k}E(\tilde v_k)
  +\frac{1-\lambda_k^{-1}}{2}\int |\nabla v_k|^2\ge
  \frac{-\nu}{\lambda_k}+\frac{1-\lambda_k^{-1}}{2}\int |\nabla
  v_k|^2. 
\end{equation*}
Doing the same for $E(w_k)$, we find
\begin{align*}
  E(v_k) + E(w_k) &\ge -\nu\( \frac{1}{\lambda_k}+\frac{1}{\mu_k}\)
  +\frac{1-\lambda_k^{-1}}{2}\int |\nabla  v_k|^2
                    +\frac{1-\mu_k^{-1}}{2}\int |\nabla  w_k|^2\\
  &\ge -\nu\( \frac{1}{\lambda_k}+\frac{1}{\mu_k}\)
  +\frac{1-\lambda_k^{-1}}{2\|v_k\|_{L^2}^2}\| v_k\|_{L^4}^4
                    +\frac{1-\mu_k^{-1}}{2\|w_k\|_{L^2}^2}\| w_k\|_{L^4}^4,
\end{align*}
where in the second step, we have used the Gagliardo-Nirenberg inequality. Passing to the
limit, yields
\begin{align*}
  \lim\inf_{k\to \infty} \( E(v_k) + E(w_k) \)\ge -\nu
  +\frac{1}{2}\min\(\frac{1-\theta}{\theta\rho},
  \frac{\theta}{(1-\theta)\rho}\) \liminf_{k\to \infty}
  \|u_{n_k}\|_{L^4}^4,
\end{align*}
and hence a contradiction to \eqref{eq:lower}, in view of the second
step and $\theta\in ]0,1[$. 
\smallbreak

\noindent {\bf Conclusion.}  At this stage, we have all the arguments
to conclude in the classical way. Assume, by contradiction, that
there exist a sequence $(u_{0,n})_{n\in \N}\subset H^1(\R^2)$, such that
\begin{equation}\label{eq:8.3.17}
  \|u_{0,n}-\phi\|_{H^1}\Tend n \infty 0,
\end{equation}
and a sequence $(t_n)_{n\in \N}\subset \R$, such that the sequence of solutions $u_n$ to \eqref{eq:nls} associated to the initial data $u_{0,n}$ satisfies
\begin{equation}\label{eq:8.3.18}
  \inf_{\varphi\in \mathcal E(\rho)}\left\|u_n(t_n, \cdot) -
    \varphi\right\|_{H^1(\R^2)}>\eps,
\end{equation}
for some $\eps>0$.
Introducing $v_n=u_n(t_n, \cdot)$, the above inequality also reads
\begin{equation*}
  \inf_{\varphi\in \mathcal E(\rho)}\|v_n-\varphi\|_{H^1(\R^2)}>\eps.
\end{equation*}
In view of \eqref{eq:8.3.17}, 
\begin{equation*}
  \int_{\R^2}|u_{0,n}|^2 \Tend n \infty \int_{\R^2}|\phi|^2,\quad
  E\(u_{0,n}\)\Tend n \infty E(\phi)=\inf_{v\in \Gamma(\rho)}E(v).
  \end{equation*}
The conservation laws for mass and energy imply
\begin{equation*}
  \int_{\R^2}|v_{n}|^2 \Tend n \infty \int_{\R^2}|\phi|^2,\quad
  E\(v_{n}\)\Tend n \infty E(\phi),
\end{equation*}
so $(v_n)_n$ is a minimizing sequence for the problem
\eqref{eq:8.3.5}. From the previous steps,   there exist a
subsequence, still denoted by $u_n$, and a sequence 
    $y_n\in \R^2$ such that $v_n(\cdot -y_n)$ has a strong limit $u$
    in $H^1(\R^2)$. In particular, $u$ satisfies \eqref{eq:8.3.5},
    hence a contradiction.


\subsection{Three-dimensional case}
\label{sec:3D}

It remains to address item (5) of Proposition~\ref{prop:soliton3D} 
\smallbreak
To this end, Theorem~4.1, (iv) from \cite{KOPV17} ensures that for $\rho$ sufficiently
large 
\begin{equation*}
  \inf\left\{ E(u);\ u\in \Gamma(\rho)\right\} =[E_{\rm min}(m),\infty[,
\end{equation*}
with $E_{\rm min}(m)<0$. It is then possible to resume the arguments
presented in Section~\ref{sec:2D} above, and obtain $\mathcal E(\rho)$-stability of 
three-dimensional solitary waves via the Cazenave-Lions argument.

\bibliographystyle{siam}

\bibliography{biblio}
\end{document}